\documentclass[titlepage,a4paper,11pt]{amsart} 
\usepackage[foot]{amsaddr}
\usepackage{amssymb}
\usepackage[lite]{amsrefs}
\usepackage{t1enc}
\usepackage[latin1]{inputenc}
\usepackage[english]{babel}
\usepackage{mathrsfs} 
\usepackage{hyperref}
\usepackage[all]{xy} 
\usepackage[lite]{amsrefs}
\usepackage[left=3cm,top=3cm,right=3cm,bottom=3cm]{geometry}

\newtheorem{thm}{Theorem}[section]
\newtheorem{prop}[thm]{Proposition}
\newtheorem{lem}[thm]{Lemma}
\newtheorem{cor}[thm]{Corollary}
\newtheorem*{defn}{Definition}
\newtheorem{rem}[thm]{Remark}
\newtheorem{ex}[thm]{Example}

\newtheorem*{nota}{Notation}

\DeclareMathOperator\Ad{Ad}
\DeclareMathOperator\Isom{Isom}
\DeclareMathOperator\Ker{Ker}
\DeclareMathOperator\tr{tr}
\DeclareMathOperator\I{Id}

\DeclareMathOperator\h{\mathcal H}
\DeclareMathOperator\B{\mathcal B}

\DeclareMathOperator\oo{{\mathcal O}}
\DeclareMathOperator\M{{\mathfrak M}}
\DeclareMathOperator\C{\mathbb{C}}
\DeclareMathOperator\R{\mathbb{R}}
\DeclareMathOperator\N{\mathbb{N}}
\DeclareMathOperator\Sim{sim}
\DeclareMathOperator\diam{diam}
\DeclareMathOperator\dist{dist}
\DeclareMathOperator\id{id}

\hypersetup{
  colorlinks   = true, 
  urlcolor     = blue, 
  linkcolor    = blue, 
  citecolor   = red    
}
\title{Geometric aspects of similarity problems}
\date{\today}
\author{Mart\'in Miglioli}\email{martin.miglioli@gmail.com}\address{IAM-CONICET. Saavedra 15, Piso 3, (1083) Buenos Aires, Argentina}
\author{Peter Schlicht}\email{peter.schlicht@epfl.ch}\address{EPFL SB MATHGEOM EGG - MA C3 584 (Batiment MA) - CH 1015 Lausanne - Switzerland}
\begin{document}
\begin{abstract}
This article presents a geometric approach to some similarity problems involving metric arguments in the non-positively curved space of positive invertible operators of an operator algebra and the canonical isometric action by invertible elements on the cone given by $g\cdot a=gag^*$.\par 
Through this approach, we extend and put into a geometric framework results by G. Pisier and partially answer a question by Andruchow, Corach and Stojanoff about minimality properties of canonical unitarizers.
\end{abstract}
\maketitle
\tableofcontents




\section{Introduction}
In this article, we address similarity problems as presented for example by G. Pisier in \cite{pisier4} or by N. Ozawa in \cite{osawa}. Generally, those problems ask, when some continuous homorphims (e.g. from a group or a $C^*$-algebra) into the algebra of bounded operators on a Hilbert space are conjugate to special homomorphisms.\par 
In the case of groups, this is Dixmier's unitarizability question: "For which groups is every uniformly bounded linear representation on a Hilbert space similar to a unitary representation". Those groups are called unitarizable. Even though this question is still open in full generality, some partial answers are known: independently Day \cite{day}, Dixmier \cite{dix} and Nakamura and Takeda \cite{naka} proved that amenable groups are unitarizable and the non-unitarizable groups were found (e.g. \cite{ehrenp}, \cite{epm}, \cite{pytls}). This led to asking whether every unitarizable group is amenable.\par 
The similarity problem in the case of $C^*$-algebras is known as Kadison's problem and partial answers were given by Haagerup, who proved in \cite{haag} that uniformly bounded cyclic representations as well as completely bounded homomorphisms are similar to \\$*$-homomorphisms.\par 
In the present article, we connect those similarity problems through induced actions on the cone of positive invertible operators with a fixed point property for actions by isometries.\par 
After introducing some terminology and background, such as the metric and geodesic structure on the cone $P$ of positive invertible elements of a $C^*$-algebra in Chapter 2, we will connect analytic data coming from a uniformly bounded map (such as its uniform bound) with geometric quantities like diameter of orbits and distances of those to fixed point sets and analyze their properties in Chapter 3. 
\par 
Chapter 4 is then devoted to putting some interpolation results of G. Pisier's studies on similarity problems into this geometric set-up. At the end of this chapter, we will restrict the focus to some subalgebras of $\B(\h)$, which allow for a CAT(0)-metric upon the same geodesic structure of their cone of positive operators to prove unitarizability results in those cases.\par 
In Chapter 5, we give a partial answer to a problem formulated by Andruchow, Corach and Stojanoff in \cites{andcorstoj1,andcorstoj2} about the minimality properties of the canonical unitarizers of some representations.\par

\section{Preliminaries}\label{prel}
\subsection{Geometry on the cone of positive invertibles}\label{geopos}

In this subsection, we recall some geometric facts about the cone $P$ of positive invertible operators in a unital $C^*$-algebra $A$. Throughout this article, we will mainly use the metric and the geodesic structure of $P$, but mention the differential geometric background for the sake of completeness.\par 
As an open subset of the real space $A_s$ of self-adjoint elements in $A$, $P$ naturally inherits the structure of a Banach manifold. It also carries a canonical symmetric space structure with Cartan symmetries given by $\sigma_a(b)=ab^{-1}a$ for $a,b\in P$, see \cite{neeb}*{Example 3.9}. The corresponding exponential map $\exp_{\id}:T_{\id}P\simeq A_s\to P$ at the identity element $\id\in P$ is given by the ordinary exponential with inverse $\log:P\to A_s$ and the corresponding geodesics between any two points $a,b\in P$ are given by
\begin{align}
\gamma_{a,b}(t)=a^{\frac12}(a^{-\frac12}ba^{-\frac12})^ta^{\frac12}.\label{geodesic}
\end{align}\par
Note that this space is isomorphic to the quotient $G/U$, where $G$ and $U$ are the groups of invertible and unitary elements in $A$.\par
By defining
\begin{align*}
d(a,b)=\left\Vert\log\left(a^{-\tfrac12}ba^{-\tfrac12}\right)\right\Vert,
\end{align*}
$P$ is turned into a metric space. This metric structure comes from a Finslerian length structure on $P$ (defined through identifying $A_s$ together with the restriction of the norm on $A$ to $A_s$ with the tangential space of $P$ at $\id$ and using the Cartan symmetries). With respect to this Finsler structure the geodesics between any two points $a,b\in P$ are length minimizing, but they are not unique as shortest continuous paths between $a$ and $b$ (since the operator norm is not uniformly convex in general). See \cite{neeb}*{Example 6.7} and \cite{correcht} for details, or \cite{schlicht}*{Chapter 4} for explicit calculations avoiding the use of Finslerian length structures.\par
We can also regard a positive invertible $a\in P$ as a Hilbertian norm $\|\cdot\|_a$ compatible with the original norm $\|\cdot\|$ and defined as  $\|\xi\|_a=\|a\xi\|=\langle a^2\xi,\xi\rangle^{\frac12}$ for $\xi\in\h$. A Banach-Mazur type distance $\delta$ on the set of norms compatible with the Hilbertian norm of $\h$  can be defined as 
\begin{align*}\delta(||\cdot||,|||\cdot|||)=\sup_{\xi\neq 0}\left|{\log\frac{||\xi||}{ |||\xi|||}}\right|.\end{align*}
It was proved in \cite{acms}*{Proposition 2.2} that $d(a,b)=\delta(\|\cdot\|_a,\|\cdot\|_b)$.
In \cite{cpr}*{Theorem 1} the "exponential metric increasing property", which states that for $a\in P$ and $X,Y\in T_aP$ the inequality $\|X-Y\|_a\leq d(\exp_a(X),\exp_a(Y))$ holds, was shown to be equivalent to Segal's inequality, which states that $\|e^{X+Y}\|\leq\|e^{\frac{X}{2}}e^Ye^{\frac{X}{2}}\|$ for self-adjoint operators $X$ and $Y$.
\par
Using the exponential metric increasing property the convexity of the distance along geodesics can be derived, see \cite{correcht}.
\begin{prop}\label{geoconv}
For two geodesics $\alpha$ and $\beta$ in $P$ the map $[0,1] \to P$, $t\mapsto d(\alpha(t),\beta(t))$ is convex, i.e.
\begin{align*}
d(\alpha(t),\beta(t))\leq t\cdot d(\alpha(1),\beta(1))+(1-t)d(\alpha(0),\beta(0))
\end{align*}

\end{prop} 
In \cite{cpr3}*{Proposition 1} or \cite{schlicht}*{Chapter 4}, the following was proved:
\begin{prop}\label{invmetricap}
The action $I$ of $G$ on $(P,d)$ given by $g\cdot a=gag^*$ is isometric. 
\end{prop}

\subsection{Basic properties of the restricted canonical action}\label{basicprop}

In this subsection, we prove the basic non-metric properties of the canonical action restricted to subgroups $H<G$.

\begin{defn}
Let $A$ be a $C^*$-algebra, $G\subseteq A$ the group of invertible elements and $P\subseteq G$ the set of positive invertibles, then for a subgroup $H< G$ we define the action $I$ of $H$ on $P$ as $h\cdot a=I_h(a)=hah^*$. If clarification is needed, we write $I_H$ to express, which group is acting.\par 
The \textit{\textbf{fixed point set}} for this action is denoted by $P^H$. The \textit{\textbf{orbit}} of $a\in P$ shall be denoted by $\oo_H(a)$. A group $H$ is said to be \textit{\textbf{unitarizable}}, if there is an invertible operator $s$ such that $s^{-1}Hs$ is a group of unitaries.  
\end{defn}

\begin{rem}\label{posunit}
Note that if $s$ is a unitarizer of $H$ and $s=bu$ its polar decomposition, then $b$ is a positive unitarizer of $H$ as $u^{-1}b^{-1}Hbu$ is a group of unitaries. In this case $\|s\|=\|b\|$.
\end{rem}

The next proposition relates positive unitarizers to fixed points of $I_H$.

\begin{prop}\label{fixed}
A positive invertible operator $s$ unitarizes $H$ if and only if $s^2$ is a fixed point of the action $I_H$.
\end{prop}

\begin{proof}
Observe that
\begin{align*}
s^{-1}Hs\subseteq U &\Leftrightarrow s^{-1}hs(s^{-1}hs)^*=\id\quad\forall h\in H\\
&\Leftrightarrow s^{-1}hs^2h^*s^{-1}=\id\quad\forall h\in H \\
&\Leftrightarrow I_h(s^2)=hs^2h^*=s^2\quad\forall h\in H. 
\end{align*}
\end{proof}

We next show how orbits and fixed point sets behave under translations. This is well-known but we include a proof for the sake of completeness.

\begin{prop}\label{tranorb}
Let a group $G$ act on a set $X$. If $H$ is a subgroup of $G$, then for $f\in G$ and $x\in X$ we have $f^{-1}\cdot\oo_H(x)=\oo_{f^{-1}Hf}(f^{-1}\cdot x)$ and $f^{-1}\cdot X^H=X^{f^{-1}Hf}$.
\end{prop}

\begin{proof}
The first identity follows from
\begin{align*}
\oo_{f^{-1}Hf}(x)&=\{(f^{-1}hf)\cdot x:h\in H\}=\{f^{-1}\cdot (h\cdot (f\cdot x)):h\in H\}  \\
&=f^{-1}\cdot \{h\cdot (f\cdot x):h\in H\} =f^{-1}\cdot\oo_H(f\cdot x). 
\end{align*}
by substitution of $f^{-1}\cdot x$ for $x$. For the second identity, we see
\begin{align*}
x\in X^{f^{-1}Hf}&\Leftrightarrow f^{-1}hf\cdot x=x\; \forall h\in H \Leftrightarrow f^{-1}\cdot (h\cdot (f\cdot x))=x\; \forall h\in H \\
&\Leftrightarrow h\cdot f\cdot x=f\cdot x\; \forall h\in H  \Leftrightarrow f\cdot x\in X^H \Leftrightarrow x\in f^{-1}\cdot X^H. 
\end{align*}
\end{proof}

\begin{rem}\label{tranorb2}
If $A$ is a $C^*$-algebra, $P$ is the set of positive invertible elements and $G$ the group of invertible elements of $A$, then Proposition \ref{tranorb} says that for a subgroup $H$ of $G$ and for $f\in G$ and $a\in P$
\begin{align*}I_{f^{-1}}(\oo_H(a))=f^{-1}\oo_H(a)f^{-1*}=\oo_{f^{-1}Hf}(f^{-1}af^{-1*})\end{align*}
and
\begin{align*}I_{f^{-1}}(p^H)=f^{-1}p^Hf^{-1*}=P^{f^{-1}Hf}.\end{align*}
\end{rem}

\begin{rem}\label{fixunit}
If $H$ is a group of unitaries in a $C^*$-algebra $A\subseteq\B(\h)$, then the commutant $H'\cap A$ of $H$ in $A$ is a  $C^*$-algebra, hence
\begin{align*}
P^H&=\{a\in P: I_h(a)=hah^{-1}=a\; \forall h\in H\} \\
&= \{a\in P: ha=ah\; \forall h\in H\} \\
&=P\cap H'=\exp(H'\cap A_s). 
\end{align*}
 
\end{rem}

\begin{defn}\label{lietrip}
A closed real subspace $S\subseteq A_s\simeq T_{\id} P$ is called a \textit{\textbf{Lie triple system}} if $[[X,Y],Z]\in S$ for every $X,Y,Z\in S$. A closed submanifold $C\subseteq P$ is \textit{\textbf{totally geodesic}} if $\exp_a(T_aC)=C$ for all $a\in C$.\par 
Here, the bracket is given by the commutator $[X,Y]=XY-YX$.
\end{defn}

\begin{prop}\label{totalgeod}
Let $H$ be a group of invertible elements, then the fixed point set $P^H$ of the action $I$ is a totally geodesic submanifold of $P$.
\end{prop}

\begin{proof}
If $H$ is not unitarizable then $P^H$ is empty. If $H$ is unitarizable and $f$ is a positive unitarizer, then by Proposition \ref{tranorb2} $P^H=fP^{f^{-1}Hf}f$, so that the fixed point set is a translation of the fixed point set of the unitary group $f^{-1}Hf$.\par 
By Remark \ref{fixunit} $P^{f^{-1}Hf}=P\cap (f^{-1}Hf)'=\exp((f^{-1}Hf)'\cap A_s)$.  Since $(f^{-1}Hf)'$ is a $*$-subalgebra of $A$, it is a Lie triple system. From the identity $[X,Y]^*=-[X^*,Y^*]$ one verifies easily that $A_s$ is in fact a Lie triple system.\par
Therefore, the intersection $(f^{-1}Hf)'\cap A_s$ is a Lie triple system and by Corollary 4.17 in \cite{condelarotonda2} $P^{f^{-1}Hf}=P\cap (f^{-1}Hf)'=\exp((f^{-1}Hf)'\cap A_s)$, being the exponential of a Lie triple system, is a totally geodesic submanifold. Since $P^H$ is a translation of the totally geodesic manifold $P^{f^{-1}Hf}$ it is also totally geodesic.
\end{proof}

\section{Metric characterization of the similarity number and size of a group}\label{simnum}

In this chapter, we define the size and similarity number of a group of invertible elements in a $C^*$-algebra and show how these quantities are related to the norm and completely bounded norm of unital homomorphism. We then show how the similarity number and size of a group depend on the diameter of orbits and the distance of orbits to fixed point sets of the canonical associated action on positive invertible elements.

\begin{defn}The \textbf{diameter} of $D\subseteq P$ is given by $\diam(D)=\sup_{x,y\in D}d(x,y)$. The \textbf{distance between two subsets} $C$ and $D\subseteq P$ is defined as $\dist(C,D)=\inf_{x\in C,y\in D}d(x,y)$. 
\end{defn}

\begin{defn}\label{sizesim}
We define the \textit{\textbf{size}} $\vert H\vert$ of a subgroup $H< G$ of $G$ by $|H|=\sup_{h\in H}\|h\|$. \par The \textit{\textbf{similarity number}} of $H$ is $\Sim(H)=\inf\{\|s\|\|s^{-1}\|:s \mbox{ is a unitarizer of }H\}$. From Remark \ref{posunit} one easily gets $\Sim(H)=\inf\{\|s\|\|s^{-1}\|:s \mbox{ is a positive unitarizer of }H\}$.
\end{defn}

In Pisier's approach to similarity problems (see \cites{pisier2,pisier3}) the similarity number and size of a group is used. Note that the similarity number defined above is not the same as the similarity degree defined by Pisier. The norm and completely bounded norm of a unital homorphism $\pi:A\to \B(\h)$ have known interpretations in terms of the size and similarity number of the group of invertibles $\pi(U_A)\subseteq G_{\B(\h)}$ which we recall.\par 
The following result is due to Haagerup, see Theorem 1.10  \cite{haag} or Theorem 9.1 and Corollary 9.2 in \cite{paulsen}. Here we use the notation $\Ad_s(a)=sas^{-1}$ and note that a bounded unital homorphism is not necessarily $*$-preserving.

\begin{thm}\label{haag}
Let $A$ be a $C^*$-algebra with unit and let $\pi:A\to \B(\h)$ be a bounded unital homomorphism. Then $\pi$ is similar to a $*$-homomorphism (i.e. there is an invertible $s\in \B(\h)$ such that $Ad_s\circ \pi$ is a $*$-homomorphism) if and only if $\pi$ is completely bounded. If $\pi$ is completely bounded then
\begin{align*}\|\pi\|_{c.b.}=\inf\{\|s^{-1}\|\|s\|:\Ad_s\circ \pi \mbox{  is a $*$-homomorphism  }\}.\end{align*}
\end{thm}
The following is Lemma 9.6 in \cite{paulsen} and follows from the fact that a $C^*$-algebra is the linear span of its unitaries.
\begin{prop}\label{orthog}
If $A$ and $B$ are unital $C^*$-algebras and $\pi:A\to B$ is a unital homomorphism, then $\pi$ is a $*$-homomorphism if and only if $\pi$ sends unitaries to unitaries, i.e. $\pi(U_A)\subseteq U_B$, where $U_A$ and $U_B$ are the group of unitaries of $A$ and $B$ respectively.\par 
Therefore, $\Ad_s\circ \pi$ is a $*$-homomorphism for an invertible $s\in B$ if and only if \newline$\Ad_{s^{-1}}(\pi(U_A))=s^{-1}\pi(U_A)s$ is a group of unitaries.
\end{prop}
Combining the previous two results we get:
\begin{prop}\label{simcb}
Let $A$ be a $C^*$-algebra with unit and let $\pi:A\to \B(\h)$ be a completely bounded unital homomorphism. Then
\begin{align*}\|\pi\|_{c.b.}=\Sim({\pi(U_A)})\end{align*}
\end{prop}
\begin{proof}
\begin{align*}
\|\pi\|_{c.b.}&=\inf\{\|s\|\|s^{-1}\|:\Ad_{s^{-1}} \circ \pi \mbox{ is a $*$-homomorphism } \}&\mbox{by Theorem \ref{haag}  } \\
&=\inf\{\|s\|\|s^{-1}\|:s \mbox{ is a unitarizer of } \pi(U_A) \}&\mbox{by Corollary \ref{orthog}} \\
&=\Sim({\pi(U_A)})&\mbox{by Definition \ref{sizesim} }
\end{align*}
\end{proof}

\begin{rem}
A corollary of the Russo-Dye theorem, which states that the closed unit ball in a unital $C^*$-algebra is the closed convex hull of unitaries (see \cite{davidson}*{Theorem I.8.4}), is the fact that if $A$ and $B$ are unital $C^*$-algebras and $\pi:A\to B$ is a unital homomorphism, then $\|\pi\|=|\pi(U_A)|$.
\end{rem}

The next theorem characterizes metrically the similarity number and size of a group $H$ of invertible operators in a $C^*$-algebra in terms of the canonical isometric action of $H$ on $P$.

\begin{thm}\label{distdiam}
For a group $H$ of invertible elements in a $C^*$-algebra the identities 
\begin{align*}\dist\left(\oo_H(\id),P^H\right)=\dist\left(\id,P^H\right)=\log(\Sim(H))\end{align*}
and
\begin{align*}\diam(\oo_H(\id))=2\log(|H|)\end{align*}
hold. 
\end{thm}

\begin{proof}

We denote by $\lambda_{\max}(a)$ and by $\lambda_{\min}(a)$ the maximum and the minimum of the spectrum of $a\in P$. Then, using the characterization of unitarizers
\begin{align}\label{ig1}
\Sim(H)&=\inf\{\|s\|\|s^{-1}\|:s \mbox{ is a positive unitarizer of }H\} \nonumber\\
&\stackrel{\textrm{Prop. \ref{fixed}}}=\inf\left\{\left\|a^{\frac12}\right\|\left\|a^{-\frac12}\right\|:a \in P^H\right\}=\inf_{a \in P^H}\left(\frac{\lambda_{\max}(a)}{\lambda_{\min}(a)}\right)^{\frac12}.
\end{align}
Also, using the fact that for $a\in P^H$ and $\alpha >0$ we have $\alpha a \in P^H$
\begin{align}\label{ig2}
\dist(\id,P^H)&=\inf_{a \in P^H}d(\id,a)=\inf_{a \in P^H}\|\log(a)\| \nonumber\\
&=\inf_{a \in P^H}\max\{\log(\lambda_{\max}(a)),-\log(\lambda_{\min}(a))\} \nonumber\\
&=\inf_{a \in P^H, \alpha > 0}\max\{\log(\lambda_{\max}(\alpha a)),-\log(\lambda_{\min}(\alpha a))\}  \\
&=\inf_{a \in P^H, c \in \mathbb{R}}\max\{\log(\lambda_{\max}(a)) + c,-\log(\lambda_{\min}(a)) -c\} \nonumber\\
&=\inf_{a \in P^H}\frac12(\log(\lambda_{\max}(a))-\log(\lambda_{\min}(a))) \nonumber\\
&=\log\left(\inf_{a \in P^H}\left(\frac{\lambda_{\max}(a)}{\lambda_{\min}(a)}\right)^{\frac12}\right). \nonumber
\end{align}
Combining (\ref{ig1}) and (\ref{ig2}) we get $\dist(\id,P^H)=\log(\Sim(H))$. Also
\begin{align*}
\dist(\oo_H(\id),P^H)&= \inf_{h\in H}\dist(I_h(\id),P^H) =\inf_{h\in H}\dist(\id,I_{h^{-1}}(P^H)) \\
&=\inf_{h\in H}\dist(\id,P^H) =\dist(\id,P^H),  
\end{align*}
where the second equality follows from the fact that $I$ is isometric, and the third equality follows from the fact that $P^H$ is $I_H$ invariant. Since 
\begin{align*}
d(\id,hh^*)&=\|\log(hh^*)\|=\max\{\log\|hh^*\|,\log\|(hh^*)^{-1}\|\}  \\
&=\max\{\log(\|h\|^2),\log(\|h^{-1}\|^2)\|\} 
\end{align*}
it follows that
\begin{align*}
\diam(\oo_H(\id))&= \sup_{h\in H} d(\id,hh^*)=\sup_{h\in H} \max\{\log(\|h\|^2),\log(\|h^{-1}\|^2)\|\} \\
&=\sup_{h\in H} \log(\|h\|^2)=\sup_{h\in H} 2\log(\|h\|) =2\log(|H|). 
\end{align*}
\end{proof}

\begin{rem}
For $H$ a unitarizable group, $h\in H$ and a positive unitarizer $s$ of $H$ one has $\|h\|=\|s(s^{-1}hs)s^{-1}\|\leq \|s\|\|s^{-1}hs\|\|s^{-1}\|= \|s\|\|s^{-1}\|$ as $s^{-1}hs$ is unitary. Taking the supremum over $h\in H$ and the infimum over positive unitarizers $s$ we obtain $|H|\leq \Sim(H).$\par 
Now, after taking logarithms this inequality turns into $D_H(\id)\leq 2\dist(\id,P^H),$ which holds for any isometric action on a metric space.\par 
Therefore, the fact that $|H|\leq \Sim(H)$ (and in particular the fact that $\|\pi\|\leq\|\pi\|_{c.b}$ for a unital homomorphism $\pi:A \to \B(\h)$) corresponds to the geometric fact that the diameter of the orbit of the identity element is less or equal than twice the distance between the identity element and the fixed point set of the action.
\end{rem}

\begin{defn}
A positive invertible $a\in P$ with spectrum $\sigma(a)$ is said to have \textit{\textbf{symmetric spectrum}}, if $\log(\max(\sigma(a)))=-\log(\min(\sigma(a)))$.\par 
Note that this is equivalent to $\|a\|=\|a^{-1}\|$.
\end{defn}

\begin{rem}\label{simdist}
Observe in the proof of $\dist(\id,P^H)=\log(\Sim(H))$ that an $a\in P^H$ minimizing the distance to $\id$ corresponds to a unitarizer $a^{\frac12}$, which minimizes the quantity $\|s\|\|s^{-1}\|$ among all unitarizers. Also, a unitarizer $s$ such that $\|s\|\|s^{-1}\|=\Sim(H)$ can be scaled to have symmetric spectrum and the resulting scaled fixed point $s^2$ minimizes the distance to $\id$. 
\end{rem}

The next lemma shows that closest points to a point $b\in P$ in $P^H$ and -by the previous remark- unitarizers realizing the similarity number exist.

\begin{lem}\label{minunit}
Let $A\subseteq \B(\h)$ be a von Neumann algebra acting on a separable Hilbert space $\h$ and let $H$ be a group of invertible operators in $A$, then for $b\in P$ there is an $a\in P^H$ such that $\dist(b,P^H)=d(b,a)$.
\end{lem} 

\begin{proof}
Using the isometric translation $c\mapsto b^{-\frac12}cb^{-\frac12}$ we can assume by Proposition \ref{tranorb2} that $b=\id$. For $a\in P$
\begin{align*}d(\id,a)=\|\log(a)\|=\max\{\log(\lambda_{\max}(a)),-\log(\lambda_{\min}(a))\},\end{align*}
where $\lambda_{\max}(a)$ and $\lambda_{\min}(a)$ denote the maximum and minimum of the spectrum of \\$a\in P\subseteq A_s$.
Hence the closed metric balls around $\id$ are operator intervals, i.e.
\begin{align*}B[\id,r]=\{b\in P:d(\id,b)\leq r\}=[e^{-r}\id,e^{r}\id]\subseteq P.\end{align*} 
There is a sequence $(a_n)_n\subseteq P^H$ such that $d(\id,a_n)\to \dist(\id,P^H)=\inf_{b \in P^H}d(\id,b)$. Since the set
$$\{a\in A:hah^*=a\; \forall h\in H\}=\bigcap_{h\in H}\{a\in A:hah^*=a\}$$
is weak operator closed, and for every $r>0$ the set $[e^{-r}\id,e^{r}\id]$ is is also weak operator closed, we conclude that $P^H\cap [e^{-r}\id,e^{r}\id]$ is weak operator closed. Also, since the weak operator topology on closed balls is metrizable and compact, it follows that there is a subsequence of $(a_n)_n$ converging weakly to an $a\in P^H$.\par 
This subsequence, which we still denote by $(a_n)_n$, also satisfies $d(\id,a_n)\to \dist(\id,P^H)$. Now, for every $\epsilon >0$ there is an $n_{\epsilon}\in \N$ such that for $n\geq n_{\epsilon}$ 
\begin{align*}a_n\in B[\id,\dist(\id,P^H)+\epsilon]=[e^{-\dist(\id,P^H)-\epsilon}\id,e^{\dist(\id,P^H)+\epsilon}\id].\end{align*} 
Since operator intervals are weak operator closed, it follows that the weak operator limit $a$ of $(a_n)_n$ is in $[e^{-\dist(\id,P^H)-\epsilon}\id,e^{\dist(\id,P^H)+\epsilon}\id]$. Therefore, $d(\id,a)<\dist(\id,P^H)+\epsilon$ for every $\epsilon>0$ so that $d(\id,a)\leq \dist(\id,P^H)$.\par
Since $d(\id,a)\geq \dist(\id,P^H)=\inf_{b \in P^H}d(\id,b)$ the conclusion follows.
\end{proof}

\section{Geometric interpolation of the similarity number and size of a group}\label{interpolation}
This chapter begins with a geometric interpolation result relating the similarity number and size of one-parameter families of groups of invertible operators. Then a translation of some of Pisier's results about the similarity degree of groups and operator algebras into our geometric set up is given and we prove a geometric interpolation result about the similarity constants of certain group extensions. We end with a subsection studying the orbit structure of isometric actions on totally geodesic sub-manifolds of positive invertible operators with stronger convexity properties.

\begin{defn}
A \textit{\textbf{geodesically convex set}} is a subset $C\subseteq P$ such that $\gamma_{a,b}(t)\in C$ for all $a,b\in C$ and $t\in [0,1]$. A map $f:P\to \R$ is called a \textit{\textbf{geodesically convex function}}, if $f(\gamma_{a,b}(t))\leq (1-t)f(a)+tf(b)$ holds for all $a,b\in C$ and $t\in [0,1]$. 
\end{defn}

\begin{nota}
For a uniformly bounded group $H$ of invertible operators in a $C^*$-algebra we denote by $D_H$ the orbit diameter function $D_H: P \to \mathbb{R}^+,~D_H(a)=\diam(\oo_H(a))$.
\end{nota}

\begin{lem}\label{lemdh}
The map $D_H: P \to \mathbb{R}^+$ is invariant for the action of $I_H$, geodesically convex and $2$-Lipschitz. 
\end{lem}

\begin{proof}
The invariance if $D_H$ follows from the fact that $\oo_H(h\cdot a)=\oo_H(a)$ for $a\in P$ and $h\in H$. To prove that $D_H$ is geodesically convex, we see that for a geodesic $\gamma_{a,b}:[0,1]\to P$ the following holds
\begin{align*}
D_H(\gamma_{a,b}(t))&=\sup_{h\in H}d(\gamma_{a,b}(t),h\cdot \gamma_{a,b}(t)) = \sup_{h\in H}d(\gamma_{a,b}(t), \gamma_{h\cdot a,h\cdot b}(t)) \\
&\stackrel{\textrm{Prop \ref{geoconv}}}\leq \sup_{h\in H}(td(a,h\cdot a) + (1-t)d(b,h\cdot b)) \\
&\leq t\sup_{h\in H}d(a,h\cdot a) + (1-t)\sup_{h\in H}d(b,h\cdot b) \\
&= tD_H(a) +(1-t)D_H(b). 
\end{align*}
the $2$-Lipschitz property for $D_H$, follows from
\begin{align*}
D_H(a)&=\sup_{h\in H}d(a,h\cdot a) \leq \sup_{h\in H}(d(a,b)+d(b,h\cdot b)+d(h\cdot b,h\cdot a)) \\
&=\sup_{h\in H}(2d(a,b)+ d(b,h\cdot b)) =2d(a,b)+\sup_{h\in H}d(b,h\cdot b) =2d(a,b)+D_H(b). 
\end{align*}
By symmetry, we get $D_H(b)-D_H(a)\leq 2d(b,a)$ so that $|D_H(a)-D_H(b)|\leq 2d(a,b)$.
\end{proof}

\begin{rem}
For a geodesic $\gamma$ in $P$, the quotient 
\begin{align*}f_{\gamma}(t)=\frac{D_H(\gamma(t))-D_H(\gamma(0))}{d(\gamma(t),\gamma(0))}\end{align*}
is a convex function of $t$ because $D_H$ is geodesically convex. It is bounded above by $2$ and bounded below by $-2$ because $D_H$ is $2$-Lipschitz.\par 
Therefore the limit of $\lim\limits_{t\to\infty}f_{\gamma}(t)$ exists and we can interpret this quantity as a slope of $D_H$ at infinity.
\end{rem}

\begin{lem}\label{lemdist} 
For a geodesically convex subset $C\subseteq P$, the map 
\begin{align*}P\to \R^+,\quad a\mapsto \dist(a,C)\end{align*}
is geodesically convex and 1-Lipschitz.
\end{lem}

\begin{proof}
Let $\epsilon >0$ and let $e,f \in C$ such that $d(a,e)<d(a,C)+\frac{\epsilon}{2}$ and $d(b,f)<d(a,C)+\frac{\epsilon}{2}$. Since $\gamma_{e,f}$ lies in $C$ and the distance along geodesics is convex (Proposition \ref{geoconv}), we have for $t\in[0,1]$
\begin{align*}
\dist(\gamma_{a,b}(t),C)&\leq \dist(\gamma_{a,b}(t),\gamma_{e,f}(t)) \leq (1-t)d(a,e)+td(b,f)\\
&\leq (1-t)\dist(a,C)+t\dist(b,C)+\epsilon , 
\end{align*}
Taking $\epsilon >0$ arbitrarily small we get the inequality. Observe also that 
\begin{align*}d(a,C)\leq \inf_{c\in C}(d(a,b)+d(b,c))=d(a,b)+d(b,C),\end{align*}
so that by symmetry we get the Lipschitz bound.
\end{proof}

\begin{rem}
The last two Propositions are valid for more general GCB-spaces (see \cite{schlicht} for details and definitions) and the Lipschitz bound is valid for arbitrary isometric actions on metric spaces.
\end{rem}

We next use the fact that the diameter of the orbit of a point and the distance of a point to a geodesically convex subset are geodesically convex functions to prove a geometric interpolation theorem. This extends results proved by Pisier using complex interpolation techniques \cite{pisier3}*{Lemmas 2.2 and 2.3}. 

\begin{thm}\label{geomint} 
If $H$ is a uniformly bounded group, $\gamma_t=\gamma_{r^2,s^2}(t)$ is the geodesic connecting positive invertible elements $r^2$ and $s^2$ and $H_t=\gamma_t^{-1/2}H\gamma_t^{1/2}$ is the one-parameter family of groups between the group $r^{-1}Hr$ and the group $s^{-1}Hs$ then
\begin{align*}|H_t|\leq |r^{-1}Hr|^{1-t}|s^{-1}Hs|^{t}\end{align*}
If $H$ is also unitarizable, then
\begin{align*}\Sim(H_t)\leq \Sim(r^{-1}Hr)^{1-t}\Sim(s^{-1}Hs)^{t}\end{align*}
In particular, if $s$ is a positive unitarizer such that $d\left(\id,P^H\right)=d(\id,s^2)$, then the one-parameter family of groups $H_t=s^{-t}Hs^t$ satisfies 
\begin{align*}|H_t|\leq |H|^{1-t}\textrm{ and }\Sim(H_t)=\Sim(H)^{1-t}.\end{align*}
\end{thm}

\begin{proof}
By Proposition \ref{invmetricap}, the action $I$ is isometric, so that by Proposition \ref{tranorb} for $f\in G$ and $b\in P$
\begin{align*}
D_{f^{-1}Hf}(b)&=\diam(\oo_{f^{-1}Hf}(b))=\diam(f^{-1}\oo_H(fbf^*)f^{-1*}) \\
&=\diam(\oo_H(fbf^*))=D_H(fbf^*). 
\end{align*} 
Now, using the fact that $\gamma_t=\gamma_{r^2,s^2}(t)$ is a geodesic and the geodesic convexity of $D_H$ proved in Lemma \ref{lemdh}
\begin{align*}
D_{H_t}(\id)&=D_{\gamma_t^{-1/2}H\gamma_t^{1/2}}(\id)=D_H(\gamma_t) \\
&=D_H(\gamma_{r^2,s^2}(t))\leq (1-t)D_H\left(r^2\right) + tD_{H}\left(s^2\right) \\
&=(1-t)D_{r^{-1}Hr}(\id) + tD_{s^{-1}Hs}(\id). 
\end{align*} 
Exponentiating this equation and using Proposition \ref{distdiam}, we get
\begin{align*}|H_t|^2\leq |r^{-1}Hr|^{2(1-t)}|s^{-1}Hs|^{2t}\end{align*}
and therefore $|H_t|\leq |r^{-1}Hr|^{1-t}|s^{-1}Hs|^{t}$.

By Proposition \ref{tranorb} and Proposition \ref{invmetricap}, we get for $f\in G$ and $b\in P$
\begin{align*}\dist\left(b,P^{f^{-1}Hf}\right)=\dist\left(b,f^{-1}P^Hf^{-1*}\right)=\dist\left(fbf^*,P^H\right).\end{align*}
Since $P^H$ is geodesically convex, we can use Lemma \ref{lemdist} to get
\begin{align*}
\dist\left(\id,P^{H_t}\right)&=\dist\left(\id,P^{\gamma_t^{-1/2}H\gamma_t^{1/2}}\right)\\
&=\dist\left(\gamma_t,P^H\right)=\dist\left(\gamma_{r^2,s^2}(t),P^H\right)\\
&\leq (1-t)d\left(r^2,P^H\right) + td\left(s^2,P^H\right) \\
&= (1-t)\dist\left(\id,P^{r^{-1}Hr}\right)+t\dist\left(\id,P^{s^{-1}Hs}\right).
\end{align*} 
Exponentiating this inequality we obtain
\begin{align*}\Sim(H_t)\leq \Sim\left(r^{-1}Hr\right)^{1-t}\Sim\left(s^{-1}Hs\right)^{t}.\end{align*}
Now, if the geodesic is $\gamma_{1,s^2}(t)=s^{2t}$, then since $H_1=s^{-1}Hs$ is a group of unitaries $|H_t|\leq |H|^{1-t}$. 

In the inequality for the similarity number we can get instead an equality: since $s^2$ is a point in $P^H$ minimizing the distance from $\id$ to $P^H$ and geodesic have minimal lenght, $s^2$ minimizes distance between the points in $P^H$ to any point in the geodesic $\gamma_{\id,s^2}(t)=s^{2t}$. Therefore
\begin{align*}\dist\left(\id,P^{H_t}\right)=\dist\left(\id,P^{\gamma_t^{-1/2}H\gamma_t^{1/2}}\right)=\dist\left(\gamma_t,P^H\right)=(1-t)\dist\left(\id,P^H\right).\end{align*}
Exponentiating this equation and using Proposition \ref{distdiam} we get $\Sim(H_t)=\Sim(H)^{1-t}$.
\end{proof}

\begin{cor}
In the case of a completely bounded unital homomorphism $\pi:A\to \B(\h)$, we can define a family of homomorphisms $\pi_t=\Ad_{s^t}\circ \pi$ such that
\begin{align*}\|\pi_t\|\leq \|\pi\|^{1-t}\mbox{  and  } \|\pi_t\|_{c.b.}=\|\pi\|_{c.b.}^{1-t}.\end{align*}
\end{cor}


Pisier used bounds that relate the similarity number and size of groups to characterize classes of groups and algebras, see Theorem 1 in \cite{pisier1} and the discussion following that theorem. If we take the logarithm in inequalities of the form 
\begin{align*}\Sim(H)\leq K|H|^{\alpha}\end{align*}
for constants $K>1$ and $\alpha>0$, we get by Proposition \ref{distdiam}
\begin{align*}\dist(\id,P^H)\leq \log(K) + \frac{\alpha}{2}D_H(\id).\end{align*}
Therefore, composing group and algebra representations restricted to unitary groups
\begin{align*}\Gamma\xrightarrow{\rho} G\to \Isom(P),\qquad U_A\xrightarrow{\pi|_{U_A}} G\to \Isom(P)\end{align*}
Theorem 1 in \cite{pisier1} has the following translations in geometric terms:

\begin{thm}\label{pisiersim} The following holds:
\begin{itemize}
\item A discrete group $\Gamma$ is amenable if and only if for every uniformly bounded representation $\rho:\Gamma\to \B(\h)$ the inequality $\dist\left(\id,P^{\rho(\Gamma)}\right)\leq D_{\rho(\Gamma)}(\id)$ holds.
\item A discrete group $\Gamma$ is finite if and only if there is a $c>0$ such that for every uniformly bounded representation $\rho:\Gamma\to \B(\h)$ the inequality $\dist(\id,P^{\rho(\Gamma)})\leq c+ \frac{1}{2}D_{\rho(\Gamma)}(\id)$ holds. 
\item A $C^*$-algebra $A$ is nuclear if and only if for every unital completely bounded homomorphism $\psi:A\to \B(\h)$  
the inequality $\dist(\id,P^{\psi(U_A)})\leq D_{\psi(U_A)}(\id)$ holds, where $U_A$ is the group of unitaries of $A$.
\item A $C^*$-algebra $A$ is finite dimensional if and only if there is a $c>0$ such that for every unital completely bounded homomorphism $\psi:A\to \B(\h)$
the inequality $\dist(\id,P^{\psi(U_A)})\leq c+ \frac{1}{2}D_{\psi(U_A)}(\id)$ holds.  
\end{itemize}
\end{thm}

In the composed actions $\Gamma\xrightarrow{\pi} G\to \Isom(P)$ if $\Gamma$ is amenable, then the distance of each point $a\in P$ to the fixed point set is less or equal than the diameter of the orbit of that point. This proposition is well known, we state and prove it in our geometric setting:

\begin{prop}\label{distamenable}
If $\pi:\Gamma\to G$ is a strongly continuous uniformly bounded representation of an amenable locally compact group $\Gamma$ and if we denote $H=\pi(\Gamma)$ then
\begin{align*}\dist\left(a,P^{\pi(\Gamma)}\right)\leq D_{\pi(\Gamma)}(a)\end{align*}
for every $a\in P$.
\end{prop}

\begin{proof}
Since the translation $b\mapsto a^{-\frac12}ba^{-\frac12}$, $P\to P$ is isometric and $g\mapsto a^{-\frac12}ga^{\frac12}$, is continuous if we endow $G$ with the strong operator topology, by Proposition \ref{invmetricap} and Proposition \ref{tranorb}
\begin{align*}\dist\left(\id,P^{a^{-\frac12}Ha^{\frac12}}\right)=\dist\left(\id,a^{-\frac12}P^Ha^{-\frac12}\right)=\dist\left(a,P^H\right)\end{align*}
and
\begin{align*}
D_{a^{-\frac12}Ha^{\frac12}}(\id)&=\diam(\oo_{a^{-\frac12}Ha^{\frac12}}(\id))=\diam(a^{-\frac12}\oo_H(a^{\frac12}\id a^{\frac12})a^{-\frac12}) \\
&=\diam(\oo_H(a^{\frac12}a^{\frac12}))=D_H(a), 
\end{align*}
so the problem is reduced to proving that $\dist(\id,P^H)\leq D_H(\id)$. Note that 
\begin{align*}\oo_H(\id)=\{hh^*:h\in H\}\subseteq [|H|^{-2}\id,|H|^2\id]\end{align*}
and that there is a fixed point $a\in P^{\pi(\Gamma)}$ in the weak closure of the linear convex hull of the orbit of $\id$, see \cite{fack}*{Proposition 1. (ii)}. Alternatively, we could have used the fixed point property characterization of amenable groups noting that $g\cdot b=\pi(g)b\pi(g)^*$ is a continuous action on the weakly compact convex set $\overline{conv}^w(\oo_H(\id))$. Hence,
\begin{align*}a \in \overline{conv}^w(\oo_H(\id))\subseteq [|H|^{-2}\id,|H|^2\id]=B[\id,2\log(|H|)],\end{align*}
so that
\begin{align*}\dist(\id,P^H)\leq d(\id,a)\leq 2\log(|H|)=D_H(\id).\end{align*}
\end{proof}

\begin{prop}\label{diagsim}
Consider a Hilbert space direct sum $\h=\oplus_{n\in \N}\h_n$, the algebra $\B(\h)$ and its diagonal subalgebra $A=\oplus_{n\in \N}\B(\h_n)$.\par 
If $H=\oplus_{n\in \N}H_n$ is a unitarizable group of invertibles in $A$, then 
\begin{align*}
\Sim_A(H)&=\min\{\|r\|\|r^{-1}\|:r\in A, r \mbox{ unitarizes }H\}\\
&=\sup_{n\in \N}\Sim_{\B(\h_n)}(H_n)=\sup_{n\in \N}\min\{\|r_n\|\|r_n^{-1}\|:r_n\in \B(\h_n), r_n \mbox{ unitarizes }H_n\}\\
&=\Sim_{\B(\h)}(H)=\inf\{\|s\|\|s^{-1}\|:s\in \B(\h), s \mbox{ unitarizes }H\}
\end{align*}
\end{prop}

\begin{proof}
If $s$ is an invertible in $\B(\h)$ unitarizing $H$, then $shs^{-1}|_{s(\h_n)}:s(\h_n)\to s(\h_n)$ is unitary for each $n\in \N$ and $h\in H$. Since $s(\h_n)$ and $\h_n$ have the same Hilbert space dimension, there is a unitary $u_n:s(\h_n)\to \h_n$. Hence, 
\begin{align*}r_n=u_n\circ s|_{h_n}:\h_n\to s(\h_n) \to \h_n\end{align*}
is a unitarizer of $H_n$ with $\|r_n\|\|r_n^{-1}\|\leq\|s\|\|s^{-1}\|$.\par
By scaling the $r_n$ so that they have symmetric spectrum we get that $r=\oplus_{n\in \N}r_n\in A$ is a unitarizer of $H$ with $\|r\|\|r^{-1}\|= \sup_{n\in \N}\|r_n\|\|r_n^{-1}\| \leq\|s\|\|s^{-1}\|$.\par 
Hence, we get $\Sim_A(H)\leq \Sim_{\B(\h)}(H)$. The inequality $\Sim_A(H)\leq \Sim_{\B(\h)}(H)$ is trivial, and it is easy to verify that $\Sim_A(H)=\sup_{n\in \N}\Sim_{\B(\h_n)}(H_n)$.
\end{proof}
The following Proposition has also been proven in \cite{pisier4} algebraically. Here, we give a more geometric and intuitive proof
\begin{thm}\label{unitconst}
Let $\Gamma$ be a locally compact unitarizable group (i.e. every strongly continuous uniformly bounded representation is conjugate to a unitary representation). Then, there are constants $K\geq 1$ and $\alpha>0$ such that for every strongly continuous uniformly bounded representations $\pi:\Gamma\to \B(\h)$  we have 
\begin{align*}\dist(\id,P^{\pi(\Gamma)})\leq \log(K) + \frac{\alpha}{2}D_{\pi(\Gamma)}(\id)\end{align*}
or equivalently
\begin{align*}\Sim(\pi(\Gamma))\leq K|\pi(\Gamma)|^\alpha.\end{align*}
\end{thm}

\begin{proof}
Assume for contradiction the conclusion not to hold. Then, there are strongly continuous uniformly bounded representations $\pi_n:\Gamma\to \B(\h_n)$ with similarity number $\Sim(\pi_n(\Gamma))>n|\pi_n(\Gamma)|^n>n~\forall n\in\N$.\par 
Now for $n\in \N$, if $|\pi_n(\Gamma)|\leq 2$, we set $\rho_n=\pi_n$. Otherwise, by Theorem \ref{geomint} we know that there is some positive invertible $s_n\in \B(\h_n)$ and a parameter $t\in [0,1]$ such that $|\Ad_{s_n^{-t}}(\pi_n(\Gamma))|\leq|\pi_n(\Gamma)|^{1-t}=2$ and $\Sim(\Ad_{s_n^{-t}}(\pi_n(\Gamma)))=\Sim(\pi_n(\Gamma))^{1-t}$. Then
\begin{align*}\Sim(\Ad_{s_n^{-t}}(\pi_n(\Gamma)))=\Sim(\pi_n(\Gamma))^{1-t}>n^{1-t}|\pi_n(\Gamma)|^{(1-t)n}=n^{1-t}2^n>n.\end{align*}
We conclude that $\rho_n=\Ad_{s^{-t}}\circ\pi_n:\Gamma \to \B(\h_n)$ is a strongly continuous uniformly bounded representation with $\Sim(\rho_n(\Gamma))>n$ and $|\rho_n(\Gamma)|\leq 2$.\par 
Now, $\rho=\oplus_{n\in \N}\rho_n:\Gamma \to \oplus_{n\in \N} \B(\h_n)\subseteq \B(\oplus_{n\in \N} \h_n)$ is a strongly continuous representation with $|\rho(\Gamma)|\leq 2$ and (by Proposition \ref{diagsim}) $\Sim_B(\rho(\Gamma))\geq\Sim_{\B(\h_n)}(\rho_n(\Gamma))>n$ for any $n\in\N$. But this contradicts the unitarizability of $\Gamma$.
\end{proof}

\begin{defn}
We call a locally compact unitarizable group $\Gamma$ \textit{\textbf{unitarizable with constants}} $(K,\alpha)$, where $K$ and $\alpha$ are the minimal constants such that the previous proposition holds.
\end{defn}

Now, we provide bounds on the constants $(K,\alpha)$ for extensions of unitarizable groups by amenable groups using metric properties of the canonical associated action:

\begin{thm}
Let $1\to \Sigma \to \Gamma \to \Lambda \to 1$ be an extension of locally compact groups such that $\Lambda$ is amenable and $\Sigma$ is unitarizable with constants $(K,\alpha)$, then $\Gamma$ is unitarizable with constants $(K^3,3\alpha +2)$.
\end{thm}

\begin{proof}
Let $\pi:\Gamma\to G$ be a strongly continuous uniformly bounded representation. Since $\Sigma$ is a normal subgroup of $\Gamma$, the set $P^{\pi(\Sigma)}$ is $\pi(\Gamma)$-invariant and the action 
\begin{align*}\Gamma\xrightarrow{\pi} G\to \Isom(P^{\pi(\Sigma)})\end{align*}
factors through an action $\Lambda\simeq\Gamma/\Sigma\xrightarrow{\overline{\pi}} G\to \Isom(P^{\pi(\Sigma)})$. By Lemma \ref{minunit} there is an $a\in P^{\pi(\Sigma)}$ minimizing the distance to $\id$, i.e. an $a$ such that $d(\id,a)=\dist(\id,P^{\pi(\Sigma)})$. Then, by an argument such as in Lemma \ref{distamenable} we get a fixed point $b\in P^{\pi(\Gamma)}$ and a bound
\begin{align*}d(a,b)\leq \diam(\oo_{\overline{\pi}(\Lambda)}(a))=\diam(\oo_{\pi(\Gamma)}(a))=D_{\pi(\Gamma)}(a).\end{align*}
Since the diameter of the orbit of a point is 2-Lipschitz (see Lemma \ref{lemdh}), $\Sigma$ has bounds $(K,\alpha)$ and $\pi(\Sigma)\subseteq \pi(\Gamma)$, we obtain
\begin{align*}
\diam(\oo_{\pi(\Gamma)}(a))&\leq2\dist(\id,a)+\diam(\oo_{\pi(\Gamma)}(\id)) \\
&=2\dist(\id,P^{\pi(\Sigma)})+\diam(\oo_{\pi(\Gamma)}(\id)) \\
&\leq2(\log(K)+\frac{\alpha}{2}\diam(\oo_{\pi(\Sigma)}(\id)))+\diam(\oo_{\pi(\Gamma)}(\id)) \\
&\leq2(\log(K)+\frac{\alpha}{2}\diam(\oo_{\pi(\Gamma)}(\id)))+\diam(\oo_{\pi(\Gamma)}(\id)). 
\end{align*}
Therefore
\begin{align*}
\dist(\id,P^{\pi(\Gamma)})&\leq d(\id,b)\leq d(\id,a)+d(a,b)\\
&\leq\log(K)+\frac{\alpha}{2}\diam(\oo_{\pi(\Gamma)}(\id))\\
&+2(\log(K)+\frac{\alpha}{2}\diam(\oo_{\pi(\Gamma)}(\id)))+\diam(\oo_{\pi(\Gamma)}(\id))\\
&\leq3\log(K)+\frac{3\alpha + 2}{2}\diam(\oo_{\pi(\Gamma)}(\id)),
\end{align*}
proving the constants for the group $\Gamma$.
\end{proof}
\subsection{Algebras with trace}\label{algtrace}
In this subsection, we consider the case of particular subalgebras $A\subseteq \B(\h)$ of $\B(\h)$, for which the family of geodesics as defined in (\ref{geodesic}) is having a stronger convexity property: there will be metrics having the same geodesics but defining a CAT(0)-structure on the corresponding subcone $P\cap A$ of positive invertible elements in $A$.\par 
One instance, where this happens, is the algebra of Hilbert-Schmidt pertubations of scalar multiples of the identity operator $A=\B_2(\h)+\C\id \subseteq \B(\h)$. 
Recall that $\B_2(\h)$ is a Banach algebra consisting of all operators such that the \textit{Hilbert-Schmidt norm} defined by $\|a\|_{HS}=\tr(a^*a)^{\frac12}$ is finite. Then, $\|a\|\leq \|a\|_{HS}$ for $a\in \B_2(\h)$.\par 
There is a natural Hilbert space structure on $A$, such that the scalar operators are orthogonal to Hilbert-Schmidt operators: $\langle a+ \alpha\id,b+\beta\id\rangle_2=\tr(ab^*)+\alpha\Bar{\beta}.$ The geometry of this space was studied in \cite{larotonda}.
Note, that $A=\B_2(\h)+\C \id$ is complete with the norm corresponding to the inner product, since the Hilbert-Schmidt inner product induces a complete norm on the ideal of Hilbert-Schmidt operators. The real space of $\B_2(\h)+\C \id$, which is $\B_2(\h)_s+\R \id$, inherits the structure of a real Banach space, and with the same inner product it becomes a real Hilbert space.\par 
Inside $\B_2(\h)_s+\R \id$, we consider the open subset $P$ of positive invertible operators. For $a\in P$, one can identify the tangential space $T_aP$ of $P$ at $a$ with $\B_2(\h)_s+\R \id$ and endow this manifold with a real Riemannian metric by means of the following formula: $\langle X,Y\rangle_a=\langle a^{-1}X,a^{-1}Y \rangle_2$.\par 
The unique geodesic $\gamma_{a,b}:[0,1]\to P$ joining $a$ and $b$ is given as in (\ref{geodesic}) by $\gamma_{a,b}(t)=a^{\frac12}(a^{-\frac12}ba^{-\frac12})^ta^{\frac12}$ and realizes the distance, i.e. $d(a,b)=\mathrm{Length}(\gamma_{a,b})=\|\log(a^{-\frac12}ba^{-\frac12})\|_2$ (see Theorem 3.8 and Remark 3.9 in \cite{larotonda}).\par 
With this metric, $P$ is a complete metric space and if $g$ is an invertible operator in $\B_2(\h)+\C \id$ then $I_g:P\to P$ is an isometry as shown in Lemma 2.5 in \cite{larotonda}.\par 
Lemma 3.11 in \cite{larotonda},
$P$ is a complete $CAT(0)$ space. 

We next prove that if a group $H$ is close in some sense to the unitary group it is unitarizable using the Bruhat-Tits fixed point theorem (see \cite{bruhattits}). 

\begin{thm}\label{hsunit}
If $H$ is a group of invertible elements in $\B_2(\h)+\C \id$ such that we have $\sup_{h\in H}\|hh^*-\id\|_2=C < \infty$ then $H$ is unitarizable.
\end{thm}

\begin{proof}
First of all, $\sup_{h\in H}\|hh^*-\id\|_2=C < \infty$ implies
\begin{align*}\|hh^*\|-\|\id\|\leq \|hh^*-\id\|\leq \|hh^*-\id\|_2\leq C~\forall h\in H,\end{align*} 
so that $(C+1)^{-1}\id \leq hh^* \leq (C+1)\id$.\par 

 We look at the induced action on $P$ and want to prove the finiteness of $\diam(\oo_H(\id))=\sup_{h\in H}\|\log(hh^*)\|_2$ to derive the existence of a fixed point.\par 
 For $h\in H$ and as $hh^*-\id$ is compact, $hh^*$ is diagonalizable and has eigenvalues $(s_j)_j\subseteq [(C+1)^{-1},(C+1)]$, hence $\|hh^*-\id\|_2^2=\sum_j(s_j-1)^2\leq C^2.$ Therefore, $\log(hh^*)$ is diagonalizable and has eigenvalues $(\log(s_j))_j$.\par 
 Now, let $D$ be a real number such that $|\log(x)|\leq D|x-1|$ for all $x\in [(C+1)^{-1},(C+1)]$.
 Then 
 \begin{align*}\|\log(hh^*)\|_2^2=\sum_j\log(s_j)^2\leq \sum D^2(s_j-1)^2 \leq D^2C^2.\end{align*}
Since the last inequality holds for all $h\in H$ we see that $\diam(\oo_H(\id))\leq D^2C^2$. Since $\oo_H(\id)$ is bounded, by Theorem \cite{lang}*{Chapter XI, Lemma 3.1 and Theorem 3.2} the circumcenter $a\in P$ of this set is a fixed point for the action of $H$. \par 
Finally, by Proposition \ref{fixed} $s=a^{\frac12}$ is a unitarizer of $H$. 
\end{proof}

\begin{cor}
If $H$ is a group of invertible elements with $\sup_{h\in H}\|h-\id\|_2=M < \infty$, then $H$ is unitarizable.
\end{cor}

\begin{proof}
That $H\subseteq \B_2(\h)+\C \id$ is apparent.
Since $\|h\|-\|\id\|\leq \|h-\id\|\leq \|h-\id\|_2\leq M$ for all $h\in H$ we see that $\|h\|\leq M+1$ for all $h\in H$. Since
\begin{align*}hh^*-\id=hh^*-h+h-\id=h(h^*-\id)+h-\id \end{align*}
for all $h\in H$ it follows that
\begin{align*}\|hh^*-\id\|_2\leq \|h\|\|h^*-\id\|_2+\|h-\id\|_2\leq (M+1)M + M\end{align*}
for all $h\in H$. By Proposition \ref{hsunit}, $H$ is unitarizable.
\end{proof}

\begin{rem}
An extension of the previous propositions to p-Schatten perturbations of scalar operators ($1<p<\infty$) is possible, as in that case we are dealing with uniformly convex non-positively curved Busemann spaces and the existence of circumcenters of bounded sets implies that the Bruhat-Tits fixed point theorem still holds. We should also remark that the aforementioned fixed point theorem holds for actions of semigroups.
\end{rem}

We want to further analyze the orbit structure of the action $I_H$. Using Proposition \ref{tranorb} we can assume that $H$ is a group of unitaries. Let $\M=(\B_2(\h)_s+\R \id)\cap H'$ denote the set of elements in $\B_2(\h)_s+\R\id$ commuting with the $H$-action. Then
\begin{align*}\B_2(\h)_s+\R \id=\M \oplus \M^{\perp}\end{align*}
with respect to the inner product $\langle x,y\rangle_2$. Note that since $\R\id\subseteq H'$ $\M^{\perp}\subseteq \B_2(\h)$ follows.\par 
The next result is Proposition 5.12 in \cite{larotonda}. It shows that the orthogonal splitting of $\B_2(\h)_s+\R \id$ also makes sense in the metric language.
\begin{prop}\label{minprop} 
If $\M\subseteq \B_2(\h)_s+\R \id$ is a Lie triple system, then the map
\begin{align*}\M^{\perp}\times \M \xrightarrow{\sim} P,\qquad (X,Y)\mapsto e^Ye^Xe^Y\end{align*}
is a diffeormorphism. If $a=e^Ye^Xe^Y$ is the decomposition of an $a\in P$, then the closest point in $\exp(\M)$ to $a$ is $e^{2Y}$, and this point is unique with this property. 
\end{prop}
The following gives a metric description of the orbit structure on $P$:
\begin{prop}\label{inv}
The sets $e^Ye^{\M^{\perp}}e^Y$ are invariant for the action $I_H$. The circumcenter of any orbit in $e^Ye^{\M^{\perp}}e^Y$ is $e^{2Y}$.
\end{prop}

\begin{proof}
We have $\M^{\perp}=\{X\in \B_2(\h)_s+\R \id: \langle X,Y\rangle_2=0 \mbox{ for all } Y\in \M\}$. Note that $\M^{\perp}$ is $\Ad_H$-invariant because if $Y+\alpha\id\in \M\subseteq \B_2(\h)_s+\R\id$, $X\in \M^{\perp}\subseteq \B_2(\h)_s$ and $h\in H$, then
\begin{align*}
\langle hXh^{-1},Y+\alpha\id \rangle_2 &=\langle hXh^{-1}+0\id,Y+\alpha\id \rangle_2=\tr((hXh^{-1})Y^*)+0\bar{\alpha}\\
&=\tr(hXY^*h^{-1})= \tr(h^{-1}hXY^*)=\tr(XY^*)= \langle X,Y+\alpha\id\rangle_2=0,
\end{align*}
where we used that $Y\in H'$ and the cyclicity property of the Hilbert-Schmidt trace.\par 
Because $\M=(\B_2(\h)_s+\R\id)\cap H'$ is the self-adjoint part of the algebra $(\B_2(\h)+\C\id)\cap H'\subseteq A$, it is a Lie triple system, see Definition \ref{lietrip} and the proof of Proposition \ref{totalgeod}. Therefore we can apply the decomposition in Proposition \ref{minprop} and write $a\in P$ as $a=e^Ye^Xe^Y$. Then, 
\begin{align*}I_h(a)=hah^{-1}=he^Ye^Xe^Yh^{-1}=e^{\Ad_h(Y)}e^{\Ad_h(X)}e^{\Ad_h(Y)}=e^Ye^{\Ad_h(X)}e^Y\end{align*}
so that the sets $e^Ye^{\M^\perp}e^Y$ are invariant for the action $I_H$. An orbit in $e^Ye^{\M^\perp}e^Y$ is of the form $\{e^Ye^{\Ad_h(X)}e^Y:h\in H\}$ for some $X\in \M^{\perp}$, and by \cite{lang}*{Chapter XI, Lemma 3.1 and Theorem 3.2} its circumcenter is a fixed point of the action which is closest to each element in the orbit. This point is $e^{2Y}$ by Proposition \ref{minprop}.
\end{proof}

Another case of a CAT(0)-space structure on positive operators are finite von Neumann algebras $A$, where the metric on $P$ is defined through the trace $\tau:A\to \C$. This case was fully treated in \cite{miglio2}

\begin{thm}
If $H$ is a group of invertible operators in a finite von Neumann algebra $A$ such that $\sup_{h\in H}\|h\| = |H| < \infty$ then there is an $s \in \{a\in A:|H|^{-1}\id\leq a \leq |H|\id\}$ such that $s^{-1}Hs$ is a group of unitary operators in $A$.
\end{thm}

\begin{defn}
If $B\subseteq A$ is a unital inclusion of $C^*$-algebras then a map $E:A\to B$ such that $E(b)=b$ for all $b\in B$ and $\|E\|=1$ is called a \textit{\textbf{conditional expectation}}. By a result of Tomiyama \cite{tomi}, these conditions imply that $E(b_1ab_2)=b_1E(a)b_2$ and $E(a^*)=E(a)^*$ for $b_1,b_2\in B$ and $a\in A$.
\end{defn}

Assume that the group $H$ is a group of unitaries in the finite von Neumann algebra $A$. By a theorem of Takesaki \cite{takesaki}, there is a conditional expectation $E:A\to H' \cap A$ compatible with the trace, i.e. $E(\tau (x))=E(x)$ for $\in A$. This conditional expectation provides an orthogonal splitting $A=(A\cap H')\oplus_{\tau} \Ker(E)$ with respect to the inner product $\langle x,y\rangle= \tau(y^*x)$. Theorem 5.4 and Corollary 5.5 in \cite{andruchowlarotonda} in this case imply the following result.
\begin{prop}\label{minvn}
Assuming the notation of the previous paragraph, let
\begin{align*}(A_s\cap \Ker(E))\times (A_s\cap H') \xrightarrow{\sim} P,\qquad (X,Y)\mapsto e^Ye^Xe^Y\end{align*}
be the bijection given by the Porta-Recht splitting (\cite{portarecht}).\par  
Then, for $a=e^Ye^Xe^Y$, the closest point in $\exp(A_s\cap H')$ to $a$ is $e^{2Y}$, and this point is unique with this property. 
\end{prop}
The invariance of normal leaves as in Proposition \ref{inv} also holds in this case and the proof uses Proposition \ref{minvn} instead of Proposition \ref{minprop}. 
\section{Minimality properties of canonical unitarizers}\label{acsproblem}

In \cite{andcorstoj1} and \cite{andcorstoj2} Andruchow, Corach and Stojanoff studied the differential geometry of spaces of representations of some classes of $C^*$-algebras and von Neumann algebras. Using the Porta-Recht splitting \cite{portarecht} and given a certain $*$-representation $\pi_0:A\to B$, for every representation $\pi_1:A\to B$ in the $G_B$-conjugacy orbit of $\pi_0$ they constructed a canonical positive invertible $e^{-X_0}$ such that $\Ad_{e^{-X_0}}\circ \pi_1$ is a $*$-representation. In Remark 5.9 in \cite{andcorstoj1} and in \cite{andcorstoj2}*{Section 1.5} it is asked if $e^{-X_0}$, satisfies $\|e^{X_0}\|\|e^{-X_0}\|=\|\pi_1\|_{c.b.}$.\par 
We give a partial answer and a geometrical insight to this question using results from the previous chapters.\par



We briefly recall some definition and constructions in \cite{andcorstoj1}. Let $A$ be a unital $C^*$-algebra and $B\subseteq \B(\h)$ be a $C^{*}$-algebra of bounded linear operators on a separable Hilbert space $\h$. Denote by $R(A,B)$ the set of bounded unital homomorphisms from $A$ to $B$ and by $R_0(A,B)$ the subset of $*$-representations. The group $G_B$ of invertible operators in $B$ acts on $R(A,B)$ by inner automorphisms by the formula
\begin{align*}(g\cdot \pi) (a)=(\Ad_g\circ \pi) (a)=g\pi(a) g^{-1}\end{align*}
for $a\in A$ and $g\in G_B$. The group of unitary operators $U_B$ acts on $R_0(A,B)$ in the same way. In this way $R(A,B)$ and $R_0(A,B)$ are homogeneous spaces.\par 
There is also an action of $U_B$ on conditional expectations defined in $B$ and given by $u\cdot E=\Ad_u\circ E \circ \Ad_{u^{-1}}$.\par
Now, for a given $\pi_0 \in R_0(A,B)$ and a fixed conditional expectation $E_{\pi_0}:B\to \pi_0(A)'\cap B$ one obtains, by the splitting theorem of Porta and Recht (\cite{portarecht}) that for every $\pi$ in the $G$-orbit of $\pi_0$ in $R(A,B)$ there is a natural way of choosing a unique positive operator $s\in G_B$ such that $\Ad_s\circ \pi$ is a $*$-representation: if $g\in G_B$ is such that $\pi_1=\Ad_g\circ \pi_0$, the Porta-Recht splitting asserts that there are $u\in U_B$, ${Y_0}={Y_0}^*\in \pi_0(A)'\cap B$ and ${Z_0}={Z_0}^*\in \Ker(E_{\pi_0})$ such that $g=ue^{Z_0}e^{Y_0}$. Then for $a\in A$
\begin{align*}
\pi_1(a)&=ue^{Z_0}e^{Y_0}\pi_0(a)e^{-{Y_0}}e^{-{Z_0}}u^*=ue^{Z_0}\pi_0(a)e^{-{Z_0}}u^* \\
&=e^{u{Z_0}u^*}u\pi_0(a)u^*e^{-u{Z_0}u^*}=e^{\Ad_u {Z_0}}(u\cdot \pi_0)(a)e^{-\Ad_u {Z_0}}. 
\end{align*}
If we define $\rho=u\cdot \pi_0=\Ad_u\circ \pi_0$, $X_0=\Ad_u({Z_0})$ and $E_{\rho}=u\cdot E_{\pi_0}= \Ad_u\circ E_{\pi_0}\circ \Ad_{u^{-1}}$, then $\Ad_{e^{-X_0}}\circ \pi_1=\rho\in R_0(A,B)$ and $X_0\in \Ker(E_{\rho})$.  This is the $X_0$ mentioned in the question by Andruchow, Corach and Recht in the introduction to this chapter.


For the next theorem, which partially answers the question from the beginning of this chapter, we need the following result by Conde and Larotonda from \cite{condelarotonda2}*{Corollary 4.39}, which we state here in the case of operator algebras and conditional expectations.
\begin{thm}\label{minexp}
Let $A$ be a $C^*$-algebra and $B$ a $C^*$-subalgebra of $A$. Let $E:A\to B$ be a conditional expectation and let 
\begin{align*}(A_s\cap \Ker(E))\times B_s \xrightarrow{\sim} P_A,\qquad (X,Y)\mapsto e^Ye^Xe^Y \\
U_A\times (A_s\cap \Ker(E))\times B_s  \xrightarrow{\sim}G_A,\qquad (u,X,Y)\mapsto ue^Xe^Y \end{align*}
be the Porta-Recht splitting of $P_A$ and of $G_A$.\par 
Then $\|(\I-E)|_{A_s}\|=1$ if and only if for every $X\in A_s\cap \Ker(E)$ and $Y\in B_s$ a closest point in $\exp(B_s)$ to $e^Ye^Xe^Y$ is $e^{2Y}$, i.e. 
\begin{align*}\dist(\exp(B_s),e^Ye^Xe^Y)=d(e^{2Y},e^Ye^Xe^Y)=\|\log(e^X)\|=\|X\|.\end{align*}
\end{thm}

\begin{rem}\label{normleaves}
The Porta-Recht splitting provides a global tubular neighborhood to the submanifold $\exp(B_s)=P\cap B_s=P_B$ and the normal leaves to this submanifold are the sets $e^Y\exp(A_s\cap \Ker(E))e^Y$ for $Y\in B_s$.
\end{rem}

\begin{thm}\label{thmacs}
Assuming the notation and construction of canonical unitarizers of the previous paragraph and assuming that for every unitarizable group $H<G_B\subseteq \B(\h)$
\begin{align*}\label{propsimsubalg}
\Sim_B(H)=\Sim_{\B(\h)}(H),
\end{align*}
then
\begin{align*}\|\pi_1\|_{c.b.}=\exp\left(\dist\left(e^{-2{X_0}},P^{\rho(U_A)}\right)\right)=\exp\left(\dist\left(e^{-2{X_0}},\exp(\rho(U_A)'\cap \B(\h)_s)\right)\right).\end{align*}
If $\|\I-E_{\pi_0}\|=1$ the equality $\|e^{X_0}\|\|e^{-{X_0}}\|=\|\pi_1\|_{c.b.}$ holds.
\end{thm}

\begin{proof}
Note that 
\begin{align*}
\|\pi_1\|_{c.b.}&=\Sim_{\B(\h)}({\pi_1(U_A)}) &\mbox{  by Proposition \ref{simcb}}\\
&=\Sim_B({\pi_1(U_A)}) &\mbox{  by assumption  }\\
&=\exp(\dist(\id,P_B^{\pi_1(U_A)}))&\mbox{  by Theorem \ref{distdiam} }  \\
&=\exp(\dist(\id,P_B^{e^{X_0}\rho(U_A)e^{-{X_0}}})) &\mbox{  since $\Ad_{e^{-X_0}}\circ \pi_1=\rho$  }\\
&=\exp(\dist(\id,e^{X_0}P_B^{\rho(U_A)}e^{X_0})) &\mbox{  by Proposition \ref{tranorb}  }  \\
&=\exp(\dist(\id,I_{e^{X_0}}(P_B^{\rho(U_A)})))\\&=\exp(\dist(I_{e^{-X_0}}(\id),P_B^{\rho(U_A)})) \\
&=\exp(\dist(e^{-2{X_0}},P_B^{\rho(U_A)})) \\
&=\exp(\dist(e^{-2{X_0}},\exp(\rho(U_A)'\cap B_s))) &\mbox{  by Remark \ref{fixunit}  }.  
\end{align*}
This proves the first equality.
Further,
\begin{align*}
E_{\rho}= \Ad_u\circ E_{\pi_0}\circ \Ad_{u^{-1}}: B\to \Ad_u(\pi_0(U_A)'\cap B)=\Ad_u(\pi_0(U_A))'\cap B=\rho(U_A)'\cap B.
\end{align*}
If $\|\I-E_{\pi_0}\|=1$, then because $\|\Ad_u\|=1$ 
\begin{align*}
\|\I-E_{\rho}\|&=\|\I-\Ad_u\circ E_{\pi_0}\circ \Ad_{u^{-1}}\|=\|\Ad_u\circ (\I-E_{\pi_0})\circ \Ad_{u^{-1}}\|\\
&\leq \|\Ad_u\|\|I-E_{\pi_0}\|\|\Ad_{u^{-1}}\|=1,
\end{align*}
so that $\|\I-E_{\rho}\|=1$. Therefore by Theorem \ref{minexp}, 
\begin{align*}\dist(\exp(\rho(U_A)'\cap B_s),e^X)=d(\id,e^X)=\|X\|\end{align*}
for all $X\in \Ker(E_{\rho})$. Hence, since $X_0\in \Ker(E_{\rho})$ 
\begin{align*}\|\pi_1\|_{c.b.}=\exp(\dist(e^{-2{X_0}},\exp(\rho(U_A)'\cap B_s)))=e^{\|2X_0\|}.\end{align*} 
We get $\|e^{X_0}\|\|e^{-{X_0}}\|\leq e^{\|{X_0}\|}e^{\|{-X_0}\|}=e^{\|2X_0\|}=\|\pi_1\|_{c.b.}$. Since $\Ad_{e^{-X_0}}\circ \pi_1$ is a $*$-homomorphism we also get $\|\pi_1\|_{c.b.}=\Sim_{\B(\h)}({\pi_1(U_A)})\leq \|e^{X_0}\|\|e^{-{X_0}}\|$.
\end{proof}

Note in the proof of the last theorem that \begin{align*}\|\pi_1\|_{c.b.}= \Sim_B({\pi_1(U_A)})= \exp\left(\dist\left(\id,P_B^{\pi_1(U_A)}\right)\right)\end{align*} and that a positive invertible $s\in B$ such that $\Ad_{s^{-1}}\circ \pi_1$ is a star representation corresponds to an $s^2\in P_B^{\pi_1(U_A)}$. After applying the isometric translation $I_{e^{-X_0}}$ to the point $\id$ and the set $P_B^{\pi_1(U_A)}$ we get the point $e^{-2{X_0}}$ and the set $P_B^{\rho(U_A)}=\exp(\rho(U_A)'\cap B_s)$. The point $e^{-2{X_0}}$ is in the leaf $\exp(\Ker(E_{\rho}\cap B_s))$ normal to the manifold $\exp(\rho(U_A)'\cap B_s))$, which contains $\id$. Therefore $\id$ is the result of projecting $e^{-2{X_0}}$ to the manifold $\exp(\rho(U_A)'\cap B_s))$ and if we apply $I_{e^{X_0}}$ we see that $e^{2X_0}$ is the result of projecting $\id$ to $P_B^{\pi_1(U_A)}$.\par 
Hence, we translated the question by Andruchow, Corach and Stojanoff into the following: Under what conditions does the point $e^{2X_0}$ minimize the distance to $\id$ among the points in $P_B^{\pi_1(U_A)}$?

\begin{rem}
Note that if $B=\oplus_{n\in \N}\B(\h_n)\subseteq \B(\h)=\B(\oplus_{n\in \N}\h_n)$ then by Proposition \ref{diagsim}, $\Sim_B(H)=\Sim_{\B(\h)}(H)$ so that $B\subseteq \B(\h)$ satisfies the condition in Theorem \ref{thmacs}.
\end{rem}

We next give an example of a conditional expectation $E$ satisfying $\|\I-E\|=1$.

\begin{ex}
Consider the Hilbert space direct sum $\h=\oplus_{n\in \N}\h_n$, the algebra $\B(\h)$ and its diagonal subalgebra $A=\oplus_{n\in \N}\B(\h_n)$. For each $n\in \N$ let $p_n$ be an orthogonal projection in $\B(\h_n)$. Let $A_n=\{X\in \B(\h_n):p_nX=Xp_n\}$ and $E_n:\B(\h_n)\to A_n$, $X\mapsto p_nXp_n+(\id-p_n)X(\id-p_n)$ be a conditional expectation as defined above.\par 
Then, $E=\oplus_{n\in \N}E_n:\oplus_{n\in \N}\B(\h_n)\to \oplus_{n\in \N}A_n$ is a conditional expectation such that $\|\I-E\|=1$. This follows from the fact that each conditional expectation $E_n$ satisfies $\|\I_n-E_n\|=1$. To see this note that $q_n=2p_n-\id$ is a self-adjoint unitary in $\B(\h_n)$ so that for every $X\in \B(\h_n)$
\begin{align*}\|(\I_n-E_n)(X)\|=\|X-E_n(X)\|=\|X-\tfrac12(X+q_nXq_n)\|=\|\tfrac12(X-q_nXq_n)\|\leq \|X\|.\end{align*}


\end{ex}

\noindent
\end{document}